\documentclass[12pt,reqno]{amsart}

\usepackage{amsmath,amssymb,amsthm,amsfonts}
\usepackage{hyperref}
\usepackage{xcolor}
\usepackage[all]{xy}
\usepackage{enumerate}

\theoremstyle{definition}
\newtheorem{definition}{Definition}[section]

\theoremstyle{plain}
\newtheorem{theorem}[definition]{Theorem}

\newtheorem{corollary}[definition]{Corollary}
\newtheorem{lemma}[definition]{Lemma}

\theoremstyle{remark}

\newtheorem{example}[definition]{Example}

\DeclareMathOperator{\OO}{O}
\DeclareMathOperator{\Rad}{Rad}

\DeclareMathOperator{\SSS}{S}
\DeclareMathOperator{\Alt}{A}
\newcommand{\K}{\mathcal{K}}
\newcommand{\CC}{\mathcal{C}}
\newcommand{\C}{\mathrm{C}}

\title{Some necessary conditions for compatibility of groups}

\author[Ding]{Zhaochen Ding}
\author[Verret]{Gabriel Verret}
\address{Department of Mathematics, University of Auckland, Private Bag 92019, Auckland 1142, New Zealand}
\email{dzha470@aucklanduni.ac.nz,g.verret@auckland.ac.nz}

\begin{document}

\begin{abstract}
    Two groups $L_1$ and $L_2$ are compatible if there exists a finite group $G$
    with isomorphic normal subgroups $N_1$ and $N_2$ such that $L_1\cong G/N_1$ and $L_2\cong G/N_2$.
    In this paper, we give new necessary conditions for two groups to be compatible.
\end{abstract}

\maketitle

\section{Introduction}
Throughout this paper, all groups are finite. We say that two groups $L_1$ and $L_2$ are \emph{compatible} if there exists a group $G$ with isomorphic normal subgroups $N_1$ and $N_2$ such that $L_1\cong G/N_1$ and $L_2\cong G/N_2$. Given two groups, one would like to determine whether they are compatible; this is the \emph{compatibility problem} for (abstract) groups.
Despite its conceptual simplicity, the compatibility problem is surprisingly challenging.
For instance, whether the alternating group $\Alt_4$ and the cyclic group $\C_{12}$ are compatible remains an open question \cite[Problem 20.21]{khukhro2024unsolvedproblemsgrouptheory}.

The compatibility problem was originally motivated by the study of subconstituents of permutation groups.
In the 1960s, interest in the following problem arose: how to determine the possible structure of primitive permutation groups with a given subconstituent.
As this problem became better understood, it was realised that there is a crucial difference between self-paired subconstituents and non-self-paired subconstituents.
This led to the following problem, which we refer to as the \emph{compatibility problem} for permutation groups: given two permutation groups, can they arise as paired subconstituents? 
Contributions to this problem include \cite{cameron1972Permutation,FAWCETT2018247,giudici2019Arctransitive, goldschmidt1980Automorphisms, knapp1973point,LI2004749, quirin1971Primitive, sims1967Graphs,weiss1974symmetrische}. 
Two regular permutation groups are compatible if and only if they are compatible as abstract groups, hence the compatibility problem for abstract groups is an instance of the compatibility problem for permutation groups.

One of the most important basic necessary condition for two groups to be compatible is \cite[Lemma 4.1]{quirin1971Primitive}, which Quirin attributes to (an unpublished result of) Sims, so we will refer to it as Sims' Lemma. In~\cite{quirin1971Primitive}, it is stated in terms of permutation groups, but we give a related version for abstract groups, see Lemma~\ref{lem:Sims}. As an immediate corollary (see Corollary~\ref{cor:compsubnormal}), one can show that compatible groups have compatible subnormal series, that is, subnormal series whose factors are the same in the same order (see Definition~\ref{def:compseries} for a more formal definition). While this is a very basic result, until now, there was no example known of groups that have compatible subnormal series but that are not compatible. Our main goal in this paper is to establish new necessary conditions for two groups to be compatible. These results allow us to show many pairs of groups are not compatible, including some that have compatible subnormal series (see Example~\ref{examp:A6Z7}). Of our main results, the simplest to state is that compatible groups with no abelian composition factors have compatible normal series (Theorem~\ref{thm:main}). This is a corollary of stronger but more technical results which can be found in Section~\ref{sec:main}.



\section{Preliminaries}



In this paper, classes of groups are closed under isomorphism and include the trivial group.

\begin{definition}
    Let $G$ be a group and let $\K$ be a class of groups.
    Define
    \[\OO^{\K}(G):=\bigcap_{A\trianglelefteq G,G/A\in\K}A\]
    and
    \[\OO_{\K}(G):=\prod_{A\trianglelefteq G,A\in\K}A.\]
    Both $\OO^{\K}(G)$ and $\OO_{\K}(G)$ are characteristic subgroups of $G$.
    We say that $\OO^{\K}(G)$ is the \emph{$\K$-residue}
    and $\OO_{\K}(G)$ is the \emph{$\K$-radical} of $G$.
\end{definition}

A \emph{Melnikov formation} is a class of groups closed under normal subgroups, quotients and extensions. It can be seen that a class of groups $\K$ is a Melnikov formation
if and only if there exists a set $\Delta$ of simple groups such that $\K$
consists of groups whose composition factors are in $\Delta$.

\begin{example}\label{examp:rad}
    If $\K$ is the class of all solvable groups, then $\K$ is a Melnikov formation and $\OO_\K(G)$
    is  denoted $\Rad(G)$.
\end{example}

In the following lemma, we introduce some basic results for Melnikov formation.
The proofs are straightforward but we include them for completeness.

\begin{lemma}\label{lem:OK}
    Let $\K$ be a Melnikov formation, let $G$ be a group, let $N,K\trianglelefteq G$ and let $f:G\rightarrow H$ be a homomorphism.
    The following statements hold.
    \begin{enumerate}[\rm (i)]
        \item  If $N,K\in\K$, then $N\cap K\in\K$ and $NK\in\K$.\label{lem:OK:i}
        \item  If $G/N\in\K$ and $G/K\in\K$, then $G/(N\cap K)\in\K$.\label{lem:OK:ii}
        \item  $\OO_\K(G)\in\K$ and $G/\OO^{\K}(G)\in\K$.\label{lem:OK:iii}
        \item  If $\OO^\K(G)\le N$, then $\OO^\K(G)=\OO^\K(N)$.\label{lem:OK:iv}
        \item  $\OO_\K(G)\cap N=\OO_\K(N)$.\label{lem:OK:v}
        \item  $f(\OO^\K(G))=\OO^\K(f(G))$. \label{lem:OK:vi}
        \item  $f(\OO_\K(G))\le \OO_\K(f(G))$.\label{lem:OK:vii}
    \end{enumerate}
\end{lemma}

\begin{proof}
    {\rm (i)} Since $N\cap K\trianglelefteq N$, we have $N\cap K\in\K$.
    Notice that $NK/(N\cap K)\cong N/(N\cap K)\times K/(N\cap K)\in\K$.
    This implies that $NK\in\K$.

    {\rm (ii)} Notice that $N/(N\cap K)\cong NK/K\in\K$ and $(G/(N\cap K))/(N/(N\cap K))\cong G/N\in\K$.
    This implies that $G/(N\cap K)\in\K$.

    {\rm (iii)}  By the definition of $\OO_\K$ 
    and $\OO^\K$, this follows from (\ref{lem:OK:i}) and (\ref{lem:OK:ii}).

    {\rm (iv)} Notice that $N/\OO^\K(G)\trianglelefteq G/\OO^\K(G)\in\K$.
    We have that $N/\OO^{\K}(G)\in\K$.
    Hence
    $\OO^\K(G)\ge \OO^\K(N)$. 
    On the other hand, notice that 
    $(G/\OO^\K(N))/(N/\OO^\K(N))\cong G/N\in\K$
    and $N/\OO^\K(N)\in\K$.
    It follows that $G/\OO^\K(N)\in\K$ because
    $G/\OO^\K(N)$ is an extension of $G/N$ by $N/\OO^\K(N)$.
    This implies that $\OO^\K(G)\le \OO^\K(N)$.

    {\rm (v)} Since $\OO_\K(G)\cap N\trianglelefteq \OO_\K(G)\in\K$,
    we have $\OO_\K(G)\cap N\in\K$.
    Also note that $\OO_\K(G)\cap N\trianglelefteq N$.
    This implies that $\OO_\K(G)\cap N\le \OO_\K(N)$. 
    On the other hand, $\OO_\K(N)\in\K$ and 
    $\OO_\K(N)\trianglelefteq G$. Hence
    $\OO_\K(N)\le \OO_\K(G)\cap N$.

    {\rm (vi)} Note that \[f(G)/f(\OO^\K(G))\cong G/(\OO^\K(G)\ker(f))\in\K.\]
    We have $f(\OO^\K(G))\ge \OO^\K(f(G))$.
    For the other direction, let $S$ be the preimage of $\OO^\K(f(G))$
    with respect to $f$. Then $G/S\cong f(G)/(\OO^\K(f(G)))\in\K$ 
    which implies that $S\ge \OO^\K(G)$.
    Hence $\OO^\K(f(G))=f(S)\ge f(\OO^\K(G))$.
    This implies that $f(\OO^\K(G))=\OO^\K(f(G))$.

    {\rm (vii)} Since $\OO_\K(G)$ is normal in $G$, we have that $f(\OO_\K(G))$
    is normal in $f(G)$. Also by (\ref{lem:OK:iii}), $f(\OO_\K(G))\in\K$.
    This implies that $f(\OO_\K(G))\le \OO_\K(f(G))$.
\end{proof}

%


\begin{theorem}
    [Wielandt's Theorem {\cite{wielandt1939Verallgemeinerung}}]\label{thm:Wielandt}
    Let $G$ be a group and let $H$ and $K$ be subnormal subgroups of G. 
    Then $\langle H,K\rangle$ is also subnormal in G.
\end{theorem}

A subnormal quasisimple subgroup of a group $G$ is called a \emph{component} of $G$.

\begin{theorem}[{\cite[6.5.2]{kurzweil2004theory}}]\label{thm:component}
 If $K$ is a component of a group $G$ and $H$ is a subnormal subgroup of $G$,   then either $K\le H$ or $[K,H]=1$.
\end{theorem}

\section{Characteristic functors and minimal witness systems}


\begin{definition}
    A \emph{characteristic functor} $F:X\mapsto F(X)$ is a map from groups to 
    subgroups such that if $\sigma:X\rightarrow Y$ is an isomorphism, then $\sigma(F(X))=F(Y)$.
\end{definition}

\begin{example}
    Let $\K$ be a class of groups. 
    The operators $\OO_\K$ and $\OO^{\K}$
    are characteristic functors.
\end{example}

\begin{definition}
    Let $L_1$, $L_2$ and $G$ be groups and let $p_1:G\rightarrow L_1$ and $p_2:G\rightarrow L_2$ be surjective homomorphisms
    such that $\ker(p_1)\cong\ker(p_2)$. We say that the triple $(G,p_1,p_2)$
    is a \emph{witness system} for the compatibility of $(L_1,L_2)$.
\end{definition}

Clearly, $L_1$ and $L_2$ are compatible if and only if there exists a witness system $(G,p_1,p_2)$
for the compatibility of $(L_1,L_2)$.

\begin{lemma}\label{lem:minimalwitness}
Let $L_1$ and $L_2$ be groups, let $(G,p_1,p_2)$ be a witness system for the compatibility of $(L_1,L_2)$ such that
    $|G|$ is minimal, let $N_1:=\ker(p_1)$, let $N_2:=\ker(p_2)$ and let $\K$ be a Melnikov formation. The following statements hold.

\begin{enumerate}[\rm (i)]
\item   If $F$ is a characteristic functor such that $F(N_1)=F(N_2)$, then $F(N_1)=1$. \label{char:functor}
\item If $p_2(N_1)\in \K$ or $p_1(N_2)\in \K$, then $N_1,N_2\in \K$. \label{lem:groupclass} 
\item  If $\OO_\K(p_2(N_1))=1$ or $\OO_\K(p_1(N_2))=1$, then $\OO_\K(N_1)=1=\OO_\K(N_2)$. \label{lem:groupclass2} 
\item  $N_1$ and $p_2(N_1)$ have the same set of  composition factors (ignoring multiplicity).  \label{cor:compfactor1}
\item $L_2$ and $G$ have the same set of composition factors (ignoring multiplicity). \label{cor:compfactor2}
\end{enumerate}
\end{lemma}

\begin{proof}\mbox{}
    \begin{enumerate}[\rm (i)]
   \item  Let $q_1:G/F(N_1)\rightarrow L_1$ and $q_2:G/F(N_2)\rightarrow L_2$
    be the homomorphisms induced by $p_1$ and $p_2$ respectively.
    We claim that $(G/F(N_1),q_1,q_2)$ is a witness system for the compatibility of $(L_1,L_2)$. 
    Since $G/F(N_1)=G/F(N_2)$, $(G/F(N_1),q_1,q_2)$ is well-defined. 
    Because $p_1$ and $p_2$ are surjective, we have that $q_1$ and $q_2$ are also surjective.
    Note that $\ker(q_1)=N_1/F(N_1)$ and $\ker(q_2)=N_2/F(N_2)$.
    Because $F$ is a characteristic functor, we have $\ker(q_1)\cong \ker(q_2)$.
    Hence \[(G/F(N_1),q_1,q_2)\] is a witness system for the compatibility of $(L_1,L_2)$.
    By the minimality of $|G|$, $F(N_1)=1$.
    
    
    \item By the definition of $\OO^{\K}$, 
    the statement is equivalent to the following:  
    \begin{center}
        if $\OO^{\K}(p_2(N_1))=1$ or $\OO^{\K}(p_1(N_2))=1$, then $\OO^{\K}(N_1)=1=\OO^{\K}(N_2)$.
    \end{center}
    Since $N_1\cong N_2$, by symmetry it is sufficient to show that $\OO^\K(N_1)=1$ when
    $\OO^\K(p_2(N_1))=1$. Assume that $\OO^\K(p_2(N_1))=1$. By Lemma \ref{lem:OK} (\ref{lem:OK:vi}), we have $p_2(\OO^\K(N_1))=1$, that is, $\OO^\K(N_1)\leq N_2$. 
    This implies  that $\OO^\K(N_1)\le N_1\cap N_2$.
    By Lemma \ref{lem:OK} (\ref{lem:OK:iv}), $\OO^\K(N_1)=\OO^\K(N_1\cap N_2)$.
    Note that  
    \[(N_1\cap N_2)/(N_1\cap \OO^\K(N_2))\cong (N_1\cap N_2)\OO^\K(N_2)/\OO^\K(N_2)\trianglelefteq N_2/\OO^\K(N_2)\in\K.\]
    Hence $(N_1\cap N_2)/(N_1\cap \OO^\K(N_2))\in \K$
    which implies that $\OO^\K(N_1\cap N_2)\le N_1\cap \OO^\K(N_2)$ and thus $\OO^\K(N_1\cap N_2)\le \OO^\K(N_2)$.
    Since $|\OO^\K(N_2)|=|\OO^\K(N_1)|=|\OO^\K(N_1\cap N_2)|$,
    we have $\OO^\K(N_2)=\OO^\K(N_1\cap N_2)=\OO^\K(N_1)$.
    It follows by (\ref{char:functor}) that $\OO^\K(N_1)=1$.

\item     Since $N_1\cong N_2$, by symmetry it is sufficient to show that $\OO_\K(N_1)=1$ when
    $\OO_\K(p_2(N_1))=1$.    Assume that $\OO_\K(p_2(N_1))=1$.
    By Lemma \ref{lem:OK} (\ref{lem:OK:vii}), we have $p_2(\OO_\K(N_1))\le \OO_\K(p_2(N_1))=1$.
    Therefore, $\OO_\K(N_1)\le N_2$ which implies that
    $\OO_\K(N_1)\le N_1\cap N_2$.
    By Lemma \ref{lem:OK} (\ref{lem:OK:v}), we have that
    $\OO_\K(N_1)=\OO_\K(N_1\cap N_2)$ and
    $\OO_\K(N_1\cap N_2)=\OO_\K(N_1)\cap N_2\le \OO_\K(N_2)$.
    Since $|\OO_\K(N_2)|=|\OO_\K(N_1)|=|\OO_\K(N_1\cap N_2)|$,
    we have $\OO_\K(N_2)=\OO_\K(N_1\cap N_2)=\OO_\K(N_1)$.
        It follows by (\ref{char:functor}) that $\OO_\K(N_1)=1$.

\item    Let $\Delta$ be the set of composition factors of $p_2(N_1)$ and let $\CC$ be the Melnikov formation consisting of groups whose composition factors are in $\Delta$.  It is clear that $p_2(N_1)\in\CC$. Applying (\ref{lem:groupclass}) with $\K=\CC$ yields that $N_1\in\CC$ and the statement easily follows.

\item   Let $\CC$ be as in the previous paragraph.  As $N_1\cong N_2$ and $N_1\in\CC$, by Lemma \ref{lem:OK} (\ref{lem:OK:i}), we have
    that $N_1N_2\in\CC$.
    Hence $N_1N_2$ and $p_2(N_1)$ have the same set of composition factors.
    Also note that $L_2/p_2(N_1)\cong G/N_1N_2$, therefore, $L_2$ and $G$ have the same set of composition factors.
\end{enumerate}
\end{proof}


\section{Sims' Lemma}

\begin{definition}
Let $L_1$ and $L_2$ be groups and
let $M_1$ and $M_2$ be normal subgroups of $L_1$ and $L_2$ respectively.
If 
\begin{itemize}
    \item [\rm (i)] $M_1$ and $M_2$ are compatible, and
    \item [\rm (ii)] $1\ne L_1/M_1\cong L_2/M_2$,
\end{itemize}
then we say that $(M_1,M_2)$ is a \emph{Sims pair} for $(L_1,L_2)$ or
that $(L_1,L_2)$ has a \emph{Sims pair} $(M_1,M_2)$.
\end{definition}


\begin{lemma}[Sims' Lemma]\label{lem:Sims}
    Let $L_1$ and $L_2$ be nontrivial compatible groups. If
    $(G,p_1,p_2)$ is a witness system for the compatibility of $(L_1,L_2)$
    such that $|G|$ is minimal, then
    \[(p_1(\ker(p_2)),p_2(\ker(p_1)))\]
    is a Sims pair for $(L_1,L_2)$.
\end{lemma}

\begin{proof}
    Let $N_1:=\ker(p_1)$, let $N_2:=\ker(p_2)$,
    let $M_1:=p_1(N_2)$ and let $M_2:=p_2(N_1)$. Clearly, $M_1$ and $M_2$ are normal subgroups of $L_1$ and $L_2$ respectively.
    Note that $M_1=p_1(N_1N_2)$ and $M_2=p_2(N_1N_2)$.
    Then $(N_1N_2,p_1|_{N_1N_2,M_1},p_2|_{N_1N_2,M_2})$
    is a witness system for the compatibility of $(M_1,M_2)$ which implies that
    $M_1$ and $M_2$ are compatible. 
    Moreover,
    \[L_1/M_1\cong G/(N_1N_2)\cong L_2/M_2.\]
    It remains only to show that $L_1/M_1\ne 1$.
    Assume $L_1/M_1=1$.
    Let $\sigma:N_2\rightarrow N_1$ be an isomorphism from $N_2$ to $N_1$.
    Let $q_1:=p_1|_{N_2,L_1}$ and let $q_2:=p_2|_{N_1,L_2}\circ \sigma$.
    We claim that $(N_2,q_1,q_2)$ is a witness
    system for the compatibility of $(L_1,L_2)$.
    Note that $L_1/M_1=1$ implies that 
    $p_1(N_2)=L_1$ and $p_2(N_1)=L_2$. It follows that  $q_1(N_2)=L_1$ and $q_2(N_2)=L_2$ that is, $q_1$ and $q_2$ are surjective.
    Also note that $\ker(q_1)=N_1\cap N_2$
    and $\ker(q_2)=\sigma^{-1}(N_1\cap N_2)$ hence $\ker(q_1)\cong\ker(q_2)$.
    Therefore $(N_2,q_1,q_2)$ is a witness system for the compatibility of $(L_1,L_2)$.
    However, this contradicts the minimality of $|G|$.
    We thereby derive $L_1/M_1\ne 1$, which concludes the proof.
\end{proof}

\begin{definition}\label{def:compseries}
    Let $L_1$ and $L_2$ be groups and
    let 
    \[S_1:1=S_{0;1}\trianglelefteq S_{1;1}\trianglelefteq\cdots\trianglelefteq S_{\ell;1}=L_1\]
    and
    \[S_2:1=S_{0;2}\trianglelefteq S_{1;2}\trianglelefteq\cdots\trianglelefteq S_{\ell;2}=L_2\]
    be subnormal series of length $\ell$ for $L_1$ and $L_2$ respectively.
    If $S_{i}/S_{i-1}\cong T_{i}/T_{i-1}$ for every $i\in\{1,\ldots, \ell\}$,
    then we say that $L_1$ and $L_2$ have \emph{compatible 
    subnormal series} $S_1$ and $S_2$.
    Analogously, we define \emph{compatible normal series}.
\end{definition}

Sims' Lemma has the following corollary.
\begin{corollary}\label{cor:compsubnormal}
    If $(L_1,L_2)$ has a Sims pair, then $L_1$ and $L_2$ have compatible subnormal
    series. In particular, if $L_1$ and $L_2$ are compatible, then they
    have compatible subnormal series.
\end{corollary}

\begin{proof}
We prove the first statement by  induction on $|L_1|$.
    Let $(M_1,M_2)$ be a Sims pair for $(L_1,L_2)$.
    If $M_1=M_2=1$, then $L_1\cong L_2$ which implies
    that $L_1$ and $L_2$ have compatible subnormal series.
    If $M_1$ and $M_2$ are nontrivial, then
    by Lemma \ref{lem:Sims}, $(M_1,M_2)$ also has a Sims pair.
    Then the induction assumption implies that $M_1$ and $M_2$
    have compatible subnormal series 
    \[1\trianglelefteq S_{1;1}\trianglelefteq \cdots\trianglelefteq S_{n;1}=M_1\]
    and
    \[1\trianglelefteq S_{1;2}\trianglelefteq \cdots\trianglelefteq S_{n;2}=M_2.\]
    Then
    \[1\trianglelefteq S_{1;1}\trianglelefteq \cdots\trianglelefteq S_{n;1}\trianglelefteq L_1\]
    and\[1\trianglelefteq S_{1;2}\trianglelefteq \cdots\trianglelefteq S_{n;2}\trianglelefteq L_2\]
    are compatible subnormal series for $L_1$ and $L_2$. This completes the proof of the first statement.
        Now we let $L_1$ and $L_2$ be compatible. By Lemma \ref{lem:Sims},
    $(L_1,L_2)$ has a Sims pair, and by the first statement, $L_1$ and $L_2$ have compatible subnormal series.
\end{proof}


\section{Main results}\label{sec:main}


\begin{definition}
    Let $T$ be a group and let $\SSS_T$ be the characteristic functor defined by
    \[\SSS_T(G):=\langle N\mid N\mbox{ is a subnormal subgroup of $G$ isomorphic to } T\rangle.\]
\end{definition}

\begin{lemma}\label{lem:FT}
    Let $T$ be a nonabelian simple group, let $G$ be a group and let $N$ be a normal subgroup 
    of $G$.  The following statements hold.
    \begin{enumerate}[\rm (i)]
        \item $\SSS_T(G)\cong T^i$ for some non-negative integer $i$.\label{lem:FT:i}
        \item $\SSS_T(G)\cap N=\SSS_T(N)$.\label{lem:FT:ii}
        \item There exists a normal subgroup $K$ of $G$ such that $K\le \SSS_T(G)$,
        $\SSS_T(G)N=K\times N$ and $K\cong T^j$ for some non-negative integer $j$.\label{lem:FT:iii}
    \end{enumerate}
\end{lemma}

\begin{proof}\mbox{}
\begin{itemize}
        \item [\rm (i)] Let $\mathcal{S}:=\{N\le G\mid N\mbox{ is a subnormal subgroup of $G$ isomorphic to } T\}$.
    Because $G$ is finite, there exists $\mathcal{H}=\{T_1,\ldots,T_n\}\subseteq\mathcal{S}$
    which is maximal with respect to the property
    that the subgroup $H$ generated by $\mathcal{H}$
    is a direct product $T_1\times\cdots\times T_n$.  By Theorem~\ref{thm:Wielandt}, $H$ is subnormal in $G$.
    Assume $S\in\mathcal{S}$ but $S\not\leq H$. Note that $S$ is  a component of $G$ and  Theorem~\ref{thm:component} implies that $[S,H]=1$ which in turn implies $S\cap H=1$. It follows that $\langle H,S\rangle=H\times S$  which contradicts the maximality of $\mathcal{S}$. This implies that $\langle S \mid S\in\mathcal{S}\rangle\leq H$ and thus $H=\SSS_T(G)$.

        \item [\rm (ii)] Subnormal subgroups of $N$ are also subnormal subgroups of $G$ hence  $\SSS_T(N)\le \SSS_T(G)$ which implies $\SSS_T(N)\le \SSS_T(G)\cap N$.
    On the other hand, by {\rm (i)}, $\SSS_T(G)\cap N\cong T^k$ for some integer $k$.
    Hence $\SSS_T(G)\cap N$ is generated by some normal subgroups of $\SSS_T(G)\cap N$,
    which are also subnormal subgroups of $N$ isomorphic to $T$.
    Hence $\SSS_T(G)\cap N\le \SSS_T(N)$. This implies that $\SSS_T(N)=\SSS_T(G)\cap N$.
       \item [\rm (iii)] By {\rm (i)}, we have $\SSS_T(G)\cong T^i$ with $T$ nonabelian simple. This implies that every normal subgroup of $\SSS_T(G)$ has a unique normal complement which must be isomorphic to $T^j$ for some integer $j$. Let $K$ be the unique normal complement of $\SSS_T(G)\cap N$   in $\SSS_T(G)$.
    Note that $\SSS_T(G)\cap N$ and $\SSS_T(G)$ are both normal in $G$, 
    hence $K$ is also normal in $G$. 
    Since \[K\cap N=(K\cap \SSS_T(G))\cap N=K\cap (\SSS_T(G)\cap N)=1,\]
    we have \[\SSS_T(G)N=(K(\SSS_T(G)\cap N))N=K((\SSS_T(G)\cap N)N)=KN=K\times N.\]
    \end{itemize}
\end{proof}

\begin{lemma}\label{lem:nonabelian-subnormal}
    Let $L_1$ and $L_2$ be groups and  let $(G,p_1,p_2)$ be a witness system for the compatibility of $(L_1,L_2)$
        such that $|G|$ is minimal. If $\ker(p_1)$ or $\ker(p_2)$ has a nonabelian simple subnormal subgroup $T$, 
        then there exist
        $M_1\trianglelefteq L_1$ and $M_2\trianglelefteq L_2$
        such that 
        $M_1\cong M_2\cong T^k$ for some  $k\geq 1$
        and $L_1/M_1$ and $L_2/M_2$
        are compatible.
\end{lemma}

\begin{proof}
Let $N_1:=\ker(p_1)$ and $N_2:=\ker(p_2)$.  Since $N_1\cong N_2$, we may assume that $N_1$ has a nonabelian simple subnormal subgroup $T$.
    Let $K_1:=\SSS_T(N_1)$, $K_2:=\SSS_T(N_2)$, and 
    let $K_0:=\SSS_T(N_1\cap N_2)$.
    By Lemma \ref{lem:FT} (\ref{lem:FT:i}) and (\ref{lem:FT:ii}),
    \[K_1\cong T^i\cong K_2\] and
    \[K_1\cap N_2=K_0=K_2\cap N_1\cong T^j\]
    for some integers $i$ and $j$. Let $k=i-j$.  If $k=0$, then $K_1=K_0=K_2$. However, this implies that $K_1=1$ by Lemma~\ref{lem:minimalwitness}(\ref{char:functor}) which
    is a contradiction, so $k\geq 1$.
    
     Let $M_1:=p_1(K_2)$ and let $M_2:=p_2(K_1)$.
    Since $K_1$ and $K_2$ are normal subgroups of $G$
    and $p_1$ and $p_2$ are surjective, $M_1$ and $M_2$
    are normal subgroups of $L_1$ and $L_2$ respectively.  We have \[M_1\cong K_2N_1/N_1\cong K_2/(K_2\cap N_1)\cong T^{i-j}\cong K_1/(K_1\cap N_2)\cong K_1N_2/N_2\cong M_2.\]

Let $\pi_1:L_1\rightarrow L_1/M_1$ and $\pi_2:L_2\rightarrow L_2/M_2$
    be the canonical projections.
    We claim that $(G,\pi_1\circ p_1,\pi_2\circ p_2)$ is a witness system for the compatibility of $(L_1/M_1, L_2/M_2)$.
    Note that $\ker(\pi_1\circ p_1)=K_2N_1$ and $\ker(\pi_2\circ p_2)=K_1N_2$.

 By Lemma~\ref{lem:FT} (\ref{lem:FT:iii}), 
    \[K_1N_2\cong T^{i-j}\times N_2\cong T^{i-j}\times N_1\cong K_2N_1.\]
    This completes the proof.    
\end{proof}

\begin{corollary}\label{cor:radcomp}
    Let $L_1$ and $L_2$ be nontrivial compatible groups.
    If 
    $\Rad(K_1)=1$ or $\Rad(K_2)=1$
    for every Sims pair $(K_1,K_2)$ of $(L_1,L_2)$,
    then there exist $M_1\trianglelefteq L_1$ and $M_2\trianglelefteq L_2$
    such that $1\ne M_1\cong M_2$, and $L_1/M_1$ and $L_2/M_2$ are compatible.

\end{corollary}

\begin{proof}
    Let $(G,p_1,p_2)$ be a witness system for the compatibility of $(L_1,L_2)$
    such that $|G|$ is minimal, let $N_1:=\ker(p_1)$ and let $N_2:=\ker(p_2)$.
    By Lemma~\ref{lem:Sims}, $(p_1(N_2),p_2(N_1))$ is a Sims pair for $(L_1,L_2)$.
    Without loss of generality,
    we can assume that \[\Rad(p_1(N_2))=1.\]
    Keeping in mind Example~\ref{examp:rad} and applying Lemma~\ref{lem:minimalwitness}(\ref{lem:groupclass2}), we obtain $\Rad(N_2)=1$. 
    If $N_2=1$, then $L_1\cong L_2$ and the conclusion holds with $M_1=L_1$ and $M_2=L_2$.
    Hence we can assume $N_2\ne 1$ and
    let $M$ be a minimal normal subgroup of $N_2$.  
    Since $\Rad(N_2)=1$,
    we have that $M\cong T^k$ for some nonabelian simple group $T$ and integer $k$. 
    Note that $T$ is subnormal in $N_2$. Applying Lemma \ref{lem:nonabelian-subnormal} completes the proof.
\end{proof}


\begin{example}\label{examp:A6Z7}
It is not hard to see that $(\Alt_6\wr\Alt_6)\times\Alt_6^6 $ has two normal subgroups isomorphic to $\Alt_6^6$ such that the respective quotients are isomorphic to $\Alt_6^7$ and $\Alt_6\wr\Alt_6$. It follows that $\Alt_6^7$ and $\Alt_6\wr\Alt_6$ are compatible.

Let $L_1:=(\Alt_6\wr\Alt_6)\times\C_7$ and let $L_2:=\Alt_6\wr\C_7=\Alt_6^7\rtimes \C_7$. The groups $L_1$ and $L_2$ are not isomorphic and the only proper non-trivial normal subgroup of $L_2$ is $K_2:=\Alt_6^7$ with $L_2/K_2\cong\C_7$. The unique normal subgroup $K_1$ of $L_1$ such that $L_1/K_1\cong\C_7$ is $K_1:=\Alt_6\wr\Alt_6$. By the previous paragraph, $K_1$ and $K_2$ are compatible hence $(K_1,K_2)$ is the unique Sims pair for $(L_1,L_2)$. Note that $\Rad(K_1)=\Rad(K_2)=1$. By Corollary \ref{cor:radcomp}, $L_1$ and $L_2$ are not compatible.
\end{example}

Example~\ref{examp:A6Z7} shows that the converse of Lemma~\ref{lem:Sims} does not hold: if $(L_1,L_2)$ has a Sims pair, then it not necessarily the case that $L_1$ and $L_2$ are compatible.

\begin{theorem}\label{thm:main}
    Compatible groups with no abelian composition factors have compatible normal series.
\end{theorem}

\begin{proof}
    Let $L_1$ and $L_2$ be compatible groups with no abelian composition factors. We must show that $L_1$ and $L_2$ have compatible normal series. We prove this by induction on $|L_1|$. Clearly this is true if $|L_1|=|L_2|=1$, so we assume this is not the case. Let $(K_1,K_2)$ be a Sims pair for $(L_1,L_2)$. Since $L_1$ and $L_2$ have no abelian composition factors, we have $\Rad(K_1)=1=\Rad(K_2)$. We may thus apply Corollary~\ref{cor:radcomp} to conclude that there exist $M_1\trianglelefteq L_1$ and $M_2\trianglelefteq L_2$
    such that $1\ne M_1\cong M_2$, and $L_1/M_1$ and $L_2/M_2$ are compatible. By induction, $L_1/M_1$ and $L_2/M_2$ have
    compatible normal series
    \[1\trianglelefteq S_{1,1}\trianglelefteq\cdots\trianglelefteq S_{n,1}=L_1/M_1\]
    and
    \[1\trianglelefteq S_{1,2}\trianglelefteq\cdots\trianglelefteq S_{n,2}=L_2/M_2.\]
For $j\in\{1,2\}$, let $\pi_j$ be the natural projection from $L_j$ to $L_j/M_j$ and, for $1\leq i\leq n$, let $\widetilde{S}_{i,j}$ be the preimage of $S_{i,j}$ with respect to $\pi_j$.
    Then $L_1$ and $L_2$ have compatible normal series
    \[1\trianglelefteq M_1\trianglelefteq \widetilde{S}_{1,1}\trianglelefteq \cdots\trianglelefteq\widetilde{S}_{n,1}=L_1\]
    and 
    \[1\trianglelefteq M_2\trianglelefteq \widetilde{S}_{1,2}\trianglelefteq \cdots\trianglelefteq\widetilde{S}_{n,2}=L_2.\]
    This completes the proof.
\end{proof}


\section*{Acknowledgements}
The authors gratefully acknowledge the financial support of the Marsden Fund UOA1824 and the University of Auckland.
\bibliography{bib}

\begin{thebibliography}{10}

\bibitem{cameron1972Permutation}
P.~J. Cameron.
\newblock Permutation groups with multiply transitive suborbits.
\newblock {\em Proc. London Math. Soc}, 3:427--440, 1972.

\bibitem{FAWCETT2018247}
Joanna~B. Fawcett, Michael Giudici, Cai~Heng Li, Cheryl~E. Praeger, Gordon Royle, and Gabriel Verret.
\newblock Primitive permutation groups with a suborbit of length 5 and vertex-primitive graphs of valency 5.
\newblock {\em Journal of Combinatorial Theory, Series A}, 157:247--266, 2018.

\bibitem{giudici2019Arctransitive}
Michael Giudici, S.P. Glasby, Cai~Heng Li, and Gabriel Verret.
\newblock Arc-transitive digraphs with quasiprimitive local actions.
\newblock {\em Journal of Pure and Applied Algebra}, 223(3):1217--1226, 2019.

\bibitem{goldschmidt1980Automorphisms}
David~M. Goldschmidt.
\newblock Automorphisms of {{Trivalent Graphs}}.
\newblock {\em Annals of Mathematics}, 111(2):377--406, 1980.

\bibitem{khukhro2024unsolvedproblemsgrouptheory}
E.~I. Khukhro and V.~D. Mazurov.
\newblock Unsolved problems in group theory. the kourovka notebook, 2024.

\bibitem{knapp1973point}
Wolfgang Knapp.
\newblock On the point stabilizer in a primitive permutation group.
\newblock {\em Mathematische Zeitschrift}, 133(2):137--168, 1973.

\bibitem{kurzweil2004theory}
Hans Kurzweil and B.~Stellmacher.
\newblock {\em The Theory of Finite Groups: An Introduction}.
\newblock Universitext. Springer, New York, 2004.

\bibitem{LI2004749}
Cai~Heng Li, Zai~Ping Lu, and Dragan Maru{\v s}i{\v c}.
\newblock On primitive permutation groups with small suborbits and their orbital graphs.
\newblock {\em Journal of Algebra}, 279(2):749--770, 2004.

\bibitem{quirin1971Primitive}
William Quirin.
\newblock Primitive {{Permutation Groups}} with {{Small Orbitals}}.
\newblock {\em Mathematische Zeitschrift}, 122:267--274, 1971.

\bibitem{sims1967Graphs}
Charles~C. Sims.
\newblock Graphs and finite permutation groups.
\newblock {\em Mathematische Zeitschrift}, 95(1):76--86, 1967.

\bibitem{weiss1974symmetrische}
RM~Weiss.
\newblock {\"U}ber symmetrische graphen, deren valenz eine primzahl ist.
\newblock {\em Mathematische Zeitschrift}, 136:277--278, 1974.

\bibitem{wielandt1939Verallgemeinerung}
Helmut Wielandt.
\newblock Eine {{Verallgemeinerung}} der invarianten {{Untergruppen}}.
\newblock {\em Mathematische Zeitschrift}, 45(1):209--244, 1939.

\end{thebibliography}
\bibliographystyle{plain}
\end{document}